\documentclass[10pt]{amsart}
 
\usepackage{a4wide}
\usepackage{nicefrac}
\usepackage{amsmath}
\usepackage{amssymb}
\usepackage{amsthm}
\usepackage{enumerate}
\usepackage{graphicx}
\usepackage{float}
\usepackage{wasysym}
\usepackage{fancyhdr}
\usepackage[toc,page]{appendix}
\usepackage{todonotes}
\usepackage{url}

\newtheorem{Lem}{Lemma}
\newtheorem{Thm}{Theorem}
\newtheorem{Prop}{Proposition}
 
\theoremstyle{definition}
\newtheorem{Rem}{Remark}
\newtheorem{alg}{Algorithm}
\newtheorem{ass}{Assumption}
\newtheorem*{Cond*}{Model Reduction Condition}

\newcommand{\R}{\mathbb{R}}
\renewcommand{\S}{\mathcal{S}}
\newcommand{\C}{\mathcal{C}}
 
\begin{document}
 
\title{Elastic Image Registration with Exact Mass Preservation}
 
 \author[Wlazlo]{Jaros\l{}aw Wlaz\l{}o}
 \address[Jaros\l{}aw Wlaz\l{}o]{\newline Transport Processes, Fraunhofer ITWM, Fraunhofer-Platz 1 Kaiserslautern, 67663, Germany}
 \email{jaroslaw.wlazlo@itwm.fraunhofer.de}
 
 \author[Fe\ss{}ler]{Robert Fe\ss{}ler}
 \address[Robert Fe\ss{}ler]{\newline Transport Processes, Fraunhofer ITWM, Fraunhofer-Platz 1 Kaiserslautern, 67663, Germany}
 \email{robert.fessler@itwm.fraunhofer.de} 
 
 \author[Pinnau]{Ren\'e Pinnau}
  \address[Ren\'e Pinnau]{\newline Technische Universit\"at Kaiserslautern, Department of Mathematics, Erwin-Schr\"odinger-Stra{\ss}e, 67663 Kaiserslautern, Germany }
  \email{pinnau@mathematik.uni-kl.de}
 
 \author[Siedow]{Norbert Siedow}
 \address[Norbert Siedow]{\newline Transport Processes, Fraunhofer ITWM, Fraunhofer-Platz 1 Kaiserslautern, 67663, Germany}
 \email{norbert.siedow@itwm.fraunhofer.de}
 
 \author[Tse]{Oliver Tse}
 \address[Oliver Tse]{\newline Technische Universit\"at Kaiserslautern, Department of Mathematics, Erwin-Schr\"odinger-Stra{\ss}e, 67663 Kaiserslautern, Germany }
 \email{tse@mathematik.uni-kl.de}
 
\begin{abstract}
We establish a new framework for image registration, which is based on linear elasticity and optimal mass transportation theory. We combine these two arguments in order to obtain a PDE constrained optimization problem that is analytically investigated and further discretized with the finite difference method and solved by an inexact SQP algorithm. This requires to solve in each step a large sparse linear system, which has a saddle point form. Motivated by stability arguments we use a fully staggered grid for the discretization of the displacement vector field. Artificial and real world examples are presented to underline the numerical robustness of the method.
\end{abstract}

\maketitle



\section{Introduction} \label{sec:introduction}
Image registration is one of the most challenging problems in image processing. Especially in medicine, which is the main direction of our interest, this process has an extremely high significance. It allows for the incorporation of different image information to facilitate and improve surgery planning or therapy procedure. Recently, as a consequence of the advanced development of imaging techniques it has become possible to perform a minimal invasive operation using real live imaging. However, in order to make the image information usable for surgeons image registration becomes essential. Moreover, an important issue is also the quality of the registration process since its results may dictate the success of a surgery.

The image registration problem may be expressed as follows. Let $R$ and $T$ be any two given images represented by their intensity functions on a domain $\Omega\subset\R^d$, where $d = 2$ corresponds to plane images and $d = 3$ to 3D images. In the following we will call $R$ the {\em reference} and $T$ the {\em template} or {\em moving} image. We wish to find a common geometric reference frame between these two images, i.e., the task is to find a transformation $\phi:\Omega\rightarrow\Omega$ such that the transformed template image $T_{\phi} := T\circ\phi$, resembles the reference image as much as possible. Naturally, this resembling can be understood in a range of different ways, depending on what kind of similarity measure we take into account. A precise mathematical formulation and physical reasoning of our choice is described in next section.

One can observe an increasing number of publications in the field of image registration over the past decades. Various models have been established and many numerical methods have been implemented \cite{MA}. The presented approaches differ from rigid to non-rigid \cite{CC2009,T}, landmarks based to intensity based \cite{VBS96,FM3}, or parametric to non-parametric ones \cite{KU2003,NPIRGEN}. The method of choice strongly depends on the specific properties of the problem. It is rather impossible to indicate an exceptional procedure which succeeds in every case. Each of them has its own advantages and drawbacks. 

A most common treatment of the task at hand is based on the following variational formulation. Find a transformation $\phi$ minimizing the total energy
\begin{equation} 	\label{op:irp}
  \mathcal{J}_{\alpha}[\phi] := \mathcal{D}(R,T_{\phi}) + \alpha \S(\phi-\phi_\text{ref}),
\end{equation}
where $\mathcal{D}$ corresponds to a dissimilarity measure quantifying the difference between the transformed template, $T_{\phi}$, and the reference, $R$. In this work we are mostly interested in mass preserving transformations \cite{LRmt}. In this case $\mathcal{D}$ takes the form
\begin{equation}  \label{op:mp}
  \mathcal{D}(R,T_\phi) := \int_{\Omega}\left(T_{\phi}\left(x\right)\det\left(\nabla\phi\left(x\right)\right)-R\left(x\right)\right)^2dx.
\end{equation}
In (\ref{op:irp}), $\S$ is a regulariying term which eliminates unwanted transformations, $\phi_\text{ref}$ is a favoured transformation (typically set to identity) and $\alpha>0$ is a regularization parameter that controls the influence of $\S$. 

In \cite{JM} (cf.~\cite{SS}) the authors showed that image registration based solely on $\mathcal{D}$ is an ill-posed problem. One can easily understand this observation by realizing that image registration is a process that tries to determine a vector field describing the transformation, while the available data is a scalar field. This is the reason for introducing $\S$ as a Tikhonov regularizing term that `convexifies' the functional in (\ref{op:irp}). Well-posedness of the regularized problem (\ref{op:irp}) depends naturally on the choice of $\mathcal{D}$ and $\S$. In most of the standard settings, one can show existence of optimal solution, see \cite{LR}. On the other hand, showing uniqueness is rather difficult or even impossible. 

Among possible choices of dissimilarity measures, there are the sum of squared differences (SSD) \cite{HHH}, mutual information \cite{MCVMS}, cross correlation \cite{LWMFS} and others. Regarding regularizers, one typically uses diffusive \cite{FM1,T}, curvature \cite{FM2,FM3} or elastic energies \cite{HM} (see also \cite{MA} and the references therein). If the imaged objects are human tissues, the elastic regularizer becomes reasonable and at times even essential. For small deformations the linear elastic model has been successfully utilized in image registration. In order to handle large deformations several approaches have been proposed. The simplest one makes use of pre-processing with a large scale affine transformation \cite{BK}. One may also use nonlinear or hyperelasticity theory \cite{gigengack2012motion,burger2013hyperelastic,LRmt}. A completely different strategy, proposed by Christensen et.~al.~\cite{CRM}, uses arguments from the theory of fluid dynamics. In this formulation, the transformation is described by the velocity field rather than the displacement. Other additional constraints may also be introduced to guarantee a desired property of the transformation, e.g., volume preservation \cite{HHM} or local rigidity \cite{JMa}.


A different but related strategy to the one described in this paper was presented by Zhu et.~al.~\cite{ZYHT}, and further investigated by Rehman in \cite{TR}. These works are based on $L^2$ optimal mass transportation \cite{LE}, where the similarity between the two given images is measured locally by a partial differential equation describing the mass preservation property. Thus the image registration problem is expressed by the constrained optimization problem 
\begin{equation} \label{eq:TC}
\begin{aligned}
  \min & \quad W(\phi):= \int_{\Omega} |\phi(x)-x|^2\rho_R(x)dx\\
  \text{subject to} & \quad \det(\nabla\phi(x))\rho_T(\phi(x)) = \rho_R(x) \quad \text{for\; $x\in \Omega$},\\
       & \quad \phi(x) = x  \quad \text{for\; $x\in\partial\Omega$},
 \end{aligned}
\end{equation}
where $\rho_R$ and $\rho_T$ represent the mass densities for $R$ and $T$ respectively, and are derived from the input images. This eliminates the need of a regularization term and makes the image registration problem completely parameter-free. Our approach is motivated by similar arguments, since we believe that the input data are true representations of the mass densities $\rho_R$ and $\rho_T$ and there is no additional mass flow in the time interval when the shots $R$ and $T$ are made. However, here, instead of the classical $L^2$ transportation cost we minimize the linear elastic strain energy \cite{MEC}, which in terms of displacement $u$ has the following form
\begin{equation} \label{elas_energy}
  \S(u) := \int_{\Omega} \frac{\mu}{4}\sum_{i,j =1}^d \left(\frac{\partial u_i}{\partial x_j}+\frac{\partial u_j}{\partial x_i}\right)^2 + \frac{\lambda}{2} \left(\sum_{i=1}^d \frac{\partial u_i}{\partial x_i}\right)^2 dx,
\end{equation} 
where $\mu$ and $\lambda$ are two positive Lam{\' e} parameters describing the mechanical properties of the imaged object. Note, that the objective functional in (\ref{eq:TC}) does not take into account derivatives of the displacements. As a consequence, the minimizing direction may lead to irregular solution. In contrast, the elastic potential energy rules out non-regular mappings, thereby providing a more smooth result.

The paper is organized as follows. In Section 2 we describe the mathematical set-up and prove an existence result for the problem. Section 3 is devoted to describing the discretization process and the assembling of the discrete analogous to the continuous image registration problem. We further provide a detailed description of the discretization in Appendix A. In Section 4 we present the optimization strategy and formulate an algorithm to solve the image registration problem. In Section 5 we present two numerical examples and finally conclude in Section 6.

\section{Mathematical formulation and existence proof} \label{sec:mf}
We now describe a parameter-free approach to the image registration problem in 2D. Let $\Omega\subset\R^2$ be a sufficiently smooth domain and $\mathcal{M}$ denote a suitable function space of deformations over $\overline{\Omega}$. We will specify the particular choice of $\mathcal{M}$ later on, but we always require $\mathcal{M}$ to be a closed subset of a reflexive Banach space $X$. Notice that due to the linear elasticity assumption we have the following relation between the total deformation $\phi$ and the displacement field $u$: $\phi = \phi_\text{ref} + u$, where $\phi_\text{ref}$ is a reference deformation for which the strain energy vanishes. In our case we naturally set $\phi_\text{ref}$ to be the identity map $id$, i.e., $\phi_\text{ref}(x) = id(x)=x$ for all $x\in\overline{\Omega}$. For any given density functions $\rho_R, \rho_T:\Omega\rightarrow\R$ having equal mass, i.e.,
\begin{equation}
  \int_{\Omega}\rho_R(x)\,dx = \int_{\Omega}\rho_T(x)\,dx,
\end{equation}
we define the following constrained optimization problem:
\begin{equation} \label{IR}
\begin{aligned}
\min\nolimits_{\phi\in X} &\quad \tilde{\S}(\phi) := \S(\phi-\phi_\text{ref})\\
  \text{subject to}                    &\quad \tilde c(\phi;\rho_R,\rho_T)  := \det(\nabla\phi)\rho_T\circ\phi-\rho_R = 0 \quad \text{in }\  \Omega\\
                          &\quad \phi(x) =  x \quad \text{on }\  \partial\Omega.
\end{aligned}\tag{P$_\phi$}
\end{equation}

In order to simplify the considerations we put here several assumptions.
\begin{ass}\label{ass0}
The domain  $\Omega$ is an open and bounded subset of $\R^2$ with Lipschitz boundary $\partial \Omega$.
\end{ass}

\begin{ass}\label{ass1}
The image intensity functions are uniformly positive, i.e., there exists $\delta>0$ such that
\begin{equation} \label{ass:rop}
  \rho_R(x) \geq \delta ,\quad \rho_T(x) \geq \delta \quad \text{for all\; $x\in\Omega$}.
\end{equation}
Additionally we require $\rho_R$ to be continuous, and $\rho_T$ to be continuously differentiable, i.e.,
\begin{equation} \label{ass:rcd}
  \rho_R\in \C^0(\overline\Omega), \qquad \rho_T\in \C^1(\Omega)\cap \C^0(\overline\Omega).
\end{equation}
\end{ass}

With the above assumptions we may reformulate the mass preserving condition in (\ref{IR}) to an equivalent form with a separated determinant part, which reads
\begin{equation}  \label{eq:cmp}
  \bar{c}(\phi;\rho_R,\rho_T) := \det(\nabla \phi) - \frac{\rho_R}{\rho_T\circ \phi}.
\end{equation}
In order to shorten the notation we will skip the second and third arguments and simply write $\bar{c}(\phi) = \bar{c}(\phi;\rho_R,\rho_T)$. It is easy to see that with (\ref{ass:rop}) $\bar{c}$ is well defined and has the same null space as the mass preserving constraint $\tilde c$ in \eqref{IR}. Let us define $\mathcal{M}$ to be a subset of $X$ that contains all feasible deformations given by
\begin{equation} \label{M}
 \mathcal{M} := \Big\{\phi\in X:\ \bar{c}(\phi) = 0 \;\; \text{in $\Omega$},\ \phi(x)=x\;\; \text{on $\partial\Omega$}\Big\}.
\end{equation}
Note, that the double symmetric gradient of the reference deformation, $\phi_\text{ref}=id$, is the identity matrix $I$ and thus $\tilde{\S}$ becomes
\begin{equation} \label{Sphi}
  \tilde{\S}(\phi) = \frac12 \int_{\Omega} \frac{\mu}{2} \Big(\nabla\phi + \nabla\phi^\top - I\Big):\Big(\nabla\phi + \nabla\phi^\top - I\Big) + \lambda \Big(\text{tr}(\nabla\phi - I)\Big)^2 dx
\end{equation}

Finally, we formulate the existence theorem for our elastic mass preserving image registration problem (\ref{IR}). It has been shown in \cite{DM} that, under certain additional assumptions, one has the existence of mass preserving mappings in $\C^1(\Omega;\R^d)$. However, in our case we additionally require the solution space to be complete with respect to a norm derived from $\S$. Therefore, we naturally choose $X$ to be the Sobolev space $X= H^1(\Omega;\R^2)$.

\begin{Thm}\label{thm:existence}
Let $\Omega\subset\R^2$ satisfy Assumption \ref{ass0} and $\rho_R$ and $\rho_T$ be two real functions on $\Omega$ satisfying Assumption \ref{ass1}. Then the problem (\ref{IR}) has at least one optimal solution in $X$.
\end{Thm}

In order to prove the above theorem we proceed according to the direct method of calculus. We split the proof into three parts: 1) weak lower semicontinuity (w.l.s.c.) of $\tilde{\S}$ in $X$, 2) coercivity of $\tilde{\S}$ on $\mathcal{M}$, 3) weak sequentially closedness of $\mathcal{M}$ in $X$. If the three conditions are satisfied, the statement of Theorem \ref{thm:existence} directly follows. 

\begin{Lem}\label{lem:wlsc}
The functional $\tilde{\S}$ defined in (\ref{Sphi}) is weakly lower semicontinuous in $X$.
\end{Lem}
\begin{proof}
It is straightforward to observe that the functional $\tilde{\S}$ is a composition of linear and convex operators and is thus convex. Additionally, it is continuous with respect to the strong topology on $X$ and in consequence also weakly lower semicontinuous.
\end{proof}

We proceed with step 2) and show that $\tilde{\S}$ satisfies
\[
 \tilde{\S}(\phi) \rightarrow \infty\quad \text{as}\quad \Vert\phi\Vert_{X} \rightarrow \infty.
\]
For $d=2$ this is a direct consequence of Korn's inequality \cite{Korn}.

\begin{Prop}[Second Korn's inequality]
 Let $\Omega\subset\R^2$ satisfy Assumption \ref{ass0}. For a displacement field $u\colon\Omega\rightarrow\R^2$ we define the linear strain tensor
	\begin{equation*}
	  \epsilon_{ij}(u):=\frac12(\partial_i u^j + \partial_j u^i),\quad i,j\in\{1,...,d\}.
	\end{equation*}
	Then, there exists a positive constant $\beta = \beta(\Omega)$ such that it holds
	\begin{equation}\label{eq:Korn}
	  \sum_{i,j = 1}^2 \left\|\epsilon_{ij}(u)\right\|^2_{L^2} \geq \beta \left\|\nabla u\right\|^2_{L^2},\quad \text{for all }u\in X_0:= \big\{u\in X:\ u=0\;\;\text{on}\;\;\partial\Omega\big\}.
	\end{equation}
\end{Prop}

\begin{Lem}\label{lem:coercive}
The functional $\tilde{\S}$ defined in (\ref{Sphi}) is coercive on $\mathcal{M}\subset X$.
\end{Lem}
\begin{proof}
Set $\mathcal{N}:= \{\phi_\text{ref} + u:\ u\in X_0\}$ and for $\phi\in\mathcal{N}$ define $u_{\phi}\in X_0$ with $\phi = u_{\phi} + \phi_\text{ref}$. Then, we have
  \begin{align*}
    \tilde{\S}(\phi) &= \int_{\Omega} \mu\sum_{i,j=1}^d \epsilon_{ij}(u_{\phi})^2 + \frac{\lambda}{2}\bigg(\sum_{i=1}^d\partial_i u_i\bigg)^2 dx 
                        \geq \mu \int_{\Omega}\sum_{i,j=1}^d \epsilon_{ij}(u_{\phi})^2 dx \geq \mu\,\beta(\Omega) \Vert u_{\phi}\Vert^2_X,
  \end{align*}
  where the last inequality follows from \eqref{eq:Korn} and the Poincar\'e inequality.
  Observe that $\mathcal{M}\subset \mathcal{N}$ and since $\phi_\text{ref}$ is a constant element in $X$ the claim follows.
\end{proof}

Next, we show that $\mathcal{M}$ is weakly sequentially closed in $X$. It is well known that for $p \geq d$ the Jacobian map
\[
 \det D:W^{1,p}(\Omega;\R^2) \rightarrow L^{p/d}(\Omega);\quad (\det D)(\phi) := \det(\nabla\phi),
\]
is continuous if we endow both spaces with the strong topology. For $p>d$ we even have continuity in the weak topology. In fact, in two dimensions the special case $p=d$ is sufficient to show existence. 

\begin{Rem}
Unfortunately when $d = 3$ the above consideration requires one to set $p>3$. However, since we wish to stay in Hilbert spaces it is necessary to require more regularity from the functional.
\end{Rem}

In the following we untilize a classical result on the weak continuity of Jacobian determinants \cite{DAC,PMS}.

\begin{Prop}\label{lem:wcjd}
  Let $\Omega\subset\R^2$ satisfy Assumption \ref{ass0}. Choose a sequence $\{\phi_k\}\subset X$ such that $\phi_k \rightharpoonup \phi$ for some $\phi$ in $X$. Then, $\det(\nabla\phi_k) \rightharpoonup^* \det(\nabla\phi)$ in the sense of measures, i.e.,
\begin{equation*}
      \int_{\Omega}\det(\nabla\phi_k)\,\psi\, dx \rightarrow \int_{\Omega}\det(\nabla\phi)\,\psi\, dx\quad \text{ for all } \psi\in \C_0^{\infty}(\Omega).
\end{equation*}
\end{Prop}

With Proposition~\ref{lem:wcjd} at hand, we can now prove weak sequentially closedness of $\mathcal{M}$.

\begin{Lem}\label{lem:wsc}
 Let $\Omega\subset\R^2$ satisfy Assumption \ref{ass0} and $\rho_R$, $\rho_T$ satisfy Assumption \ref{ass1}. Then, the mass preserving manifold $\mathcal{M}$ is weakly sequentially closed in $X$.
\end{Lem}

\begin{proof}
Let $\{\phi_k\}\subset\mathcal{M}$ be a sequence such that $\phi_k\rightharpoonup \phi$ for some $\phi\in X$. Due to the results in \cite{DM}, $\mathcal{M}$ is non-empty and thus such a  sequence exists. We want to show that $\phi\in\mathcal{M}$, i.e., it holds $\det(\nabla\phi) = \rho_R/\rho_T\circ \phi$ in $\Omega$ and $\phi = \phi_\text{ref}$ on $\partial\Omega$.

  The second condition is trivial, hence let us prove the first one. Due to Sobolev's embedding theorem, we have the strong convergence $\phi_k \rightarrow \phi$ in $L^2(\Omega;\R^2)$. Consequently, $\{\phi_k\}$ possesses a a.e.~ converging subsequence $\{\phi_{k_i}\}\subset \{\phi_k\}$ such that $\phi_{k_i}(x) \rightarrow \phi(x)$ for a.e.~$x\in\Omega$. From Assumption~\ref{ass1} we further obtain
  \begin{equation*}
    \rho_R/\rho_T\circ\phi_{k_i} =: F_i \longrightarrow F := \rho_R/\rho_T\circ\phi \quad \text{ for a.e.~$x\in\Omega$}.
  \end{equation*}
  Since $\{\phi_{k_i}\}\subset\mathcal{M}$ we additionally have
  \begin{equation*}
    \det(\nabla\phi_{k_i}) = F_i \quad \text{a.e.~in $\Omega$},
  \end{equation*}
  and thus $\det(\nabla\phi_{k_i}) \rightarrow F$  for a.e.~$x\in\Omega$.
  
  Observe that the functions $\det(\nabla\phi_{k_i})$ are uniformly  bounded. Indeed, due to Assumption~\ref{ass1},
  \[
   |\det(\nabla\phi_{k_i})| = |\rho_R/\rho_T\circ\phi_{k_i}| \le (1/\delta)\|\rho_R\|_\infty \quad\text{for all\; $i\in\mathbb{N}$}.
  \]
  Therefore, we have that
  \begin{equation*}
    \det(\nabla\phi_{k_i}) \rightharpoonup F \quad \text{in }\ L^p(\Omega;\R^2),\ \ 1\leq p <\infty.
  \end{equation*}
  In particular, the above holds for $L^2(\Omega,\R^2)$, i.e.,
  \begin{equation*}
    \quad \int_{\Omega} \det(\phi_{k_i}(x))\psi(x),\ dx \rightarrow \int_{\Omega}F(x)\psi(x)\, dx\quad \text{ for all } \psi\in L^2(\Omega;\R^2).
  \end{equation*}
  Hence, Proposition~\ref{lem:wcjd} gives $F = \det(\nabla\phi)$ which concludes the proof.
\end{proof}

Finally, we can prove Theorem \ref{thm:existence}.

\begin{proof}[Proof of Theorem \ref{thm:existence}]
  Let $\{\phi_k\}$ be a minimizing sequence of $\tilde\S$ in $\mathcal{M}$, i.e., 
  \begin{equation*}
  \lim\nolimits_{k\rightarrow\infty}\tilde\S(\phi_k) = \inf\nolimits_{\phi\in\mathcal{M}}\tilde\S(\phi).
  \end{equation*}
  Due to coercivity (Lemma \ref{lem:coercive}), this sequence must be bounded. Therefore, there is a subsequence $\{\phi_{k_i}\}\subset \{\phi_k\}$ and $\phi\in X$ such that $\phi_{k_i}\rightharpoonup \phi$. Due to weak sequential closedness of $\mathcal{M}$ we deduce $\phi\in\mathcal{M}$. Finally, the weak lower semicontinuity of $\tilde\S$ (Lemma \ref{lem:wlsc}) ensures that $\phi$ is a desired global minimizer.
\end{proof}

In the following, we consider an equivalent formulation of \eqref{IR} for displacement fields given by
\begin{equation} \label{UR}
\begin{aligned}
\min\nolimits_{u\in X_0} &\quad \S(u) \\
\text{subject to} &\quad c(u) :=\tilde c(\phi_\text{ref}+u;\rho_R,\rho_T) = 0 \quad \text{in }\  \Omega.
\end{aligned}\tag{P$_u$}
\end{equation}

\begin{Rem}
As mentioned, in the 3D case the Jacobian determinant involved in the constraints creates some difficulties. One can either leave the Hilbert setting and prove existence in $W^{1,p}$, with $p\ge 3$ or stay in Hilbert space but require more regularity. This can be done, for example, by adding a second-order term to the objective functional. 
\end{Rem}


\section{Discretization}\label{sec:disc}
Our numerical method attempts to solve the optimization problem \eqref{UR} by using a staggered finite difference (FD) discretization on a uniform grid, which is a common practice in computational fluid dynamics (CFD) \cite{CFD}. The choice of regular grid is in some sense determined by the input image data which, in fact, corresponds to a regular grid of pixel intensities. Moreover, the use of a regular lattice highly reduces the required storage. Since image registration is often applied to very high resolution images, the above argument can be important. 
Additionally, in general FD methods lead to sparser stencils than their finite element counterpart. Finally, it is well known in CFD that non-staggered discretization schemes may lead to grid-scale oscillations. Since our discrete problem resembles the stationary Navier--Stokes equations, we discretize \eqref{UR} on a staggered grid in order to stabilize the numerical method and prevent undesired oscillations. The complete and precise discretization process is described in Appendix \ref{app:discretization}.

In this numerical section, we choose the discretize-then-optimize strategy \cite{MDG}, i.e., we first discretize the continuous constrained optimization problem \eqref{UR} and then derive the required discrete optimality conditions.
The solution procedure for the resulting nonlinear system isdescribed in the next section. 

Here, we simply clarify that $\Omega_c$ and $\Omega_u$ are discrete computational domains of cell centered and edge centered grid points. By $H_c$ and $H_u$ we denote the discrete function spaces defined on $\Omega_c$ and $\Omega_u$, respectively (cf.~Appendix~\ref{subsec:staggered_grid}). Moreover, let $c^h$ and $S^h$ denote the constraint and objective functional that act on functions in $H_u$ and yield values in $H_c$ and $\mathbb{R}$, respectively (cf.~Appendix~\ref{sec:dc} and \ref{sec:dr}). With these notations, we obtain a finite dimensional equality constrained optimization problem with $n = |\Omega_u|$ unknowns and $m = |\Omega_c|$ constraints:
\begin{equation} \label{eq:IRDopt}
\begin{aligned}
  \min\nolimits_{u^h\in H_u} & \quad S^h(u^h)\\
  \text{subject to}  & \quad c^h(u^h) = 0,
\end{aligned}\tag{P$_h$}
\end{equation}
where the boundary conditions are already included in $S^h$ and $c^h$. One can easily observe that $S^h$ and $c^h$ are both polynomials in $u^h$ and thus continuously differentiable. Letting the variable $p^h\in H_c$ denote the Lagrange multiplier, we construct the Lagrangian function as
\begin{equation}
  L(u^h,p^h) = S^h(u^h) + \langle p^h, c^h(u^h)\rangle_{H_c},
\end{equation}
where $\langle\cdot,\cdot \rangle_{H_c}$ denotes the inner product in the Hilbert space $H_c$. The classical Karush--Kuhn--Tucker (KKT) conditions for optimality then read
\begin{equation}
\begin{aligned}
  d_{u}L(u^h,p^h)[v] &:= \delta S^h(u^h)[v] + \langle p^h, \delta c^h(u^h)[v]\rangle_{H_c} \quad & \text{ for all } v\in H_u,\\
  d_{p}L(u^h,p^h)[q] &:= \langle q,c^h(u^h)\rangle_{H_u}\quad  &\text{ for all } q\in H_c,
\end{aligned}
\end{equation}
where the derivatives $\delta S^h(u^h)\in H_u^*$ and $\delta c^h(u^h)\in \mathcal{L}(H_u,H_c)$ are defined in (\ref{eq:dS}) and (\ref{eq:dc}), respectively. Using the Riesz representation theorem, the above system reduces to
\begin{equation} \label{eq:IRDopt_cond}
\begin{aligned}
  A^hu^h + \textbf{B}^{h,\top}(u^h)p^h &=0,\\
  c^h(u^h)&=0,
\end{aligned}\tag{OC$_h$}
\end{equation}
with $A^h$ defined in (\ref{eq:A}) and $\textbf{B}^h(u^h)$ being the $m\times n$ Jacobian matrix of $c^h$ evaluated at $u^h$.

\begin{Rem}
One easily notices that the divergence operator $\nabla\cdot$ is a linearization of the Jacobian determinant $\det D$. Additionally, $A^h$ stands for the elliptic Navier--Lam{\' e} operator and thus with $\rho_T \equiv 1$ the system (\ref{eq:IRDopt_cond}) mimics the Stokes system.
\end{Rem}

\section{Optimization} \label{sec:opt}
In the following present the optimization algorithm. As it is common in image registration we follow the multiresolution strategy $rl=0\rightarrow rl=-1 ... \rightarrow rl=-rl_{min}$, where at each lower resolution level, $rl$, we have to solve a smaller optimization problem (cf.~Section~\ref{sec:mrs}). This problem is solved iteratively by an inexact SQP algorithm, where the initial guess is taken from the solution at the lower resolution level. At each stage of the iterative process we need to solve a possibly large KKT system and this is done by a Krylov subspace method with a suitable preconditioner.

\subsection{One level image registration}
For a fixed resolution level $rl$ let $\rho_R^{rl}$, $\rho_T^{rl}$ be the input mass density functions expected to be registered. The resolution of images determine uniquely the shape of discrete domains $\Omega_c$ and $\Omega_u$ where we set up the discrete analogue of the objective function to be minimized and the constraints (cf.~Section~\ref{sec:disc}). Thus we have a nonlinear optimization problem of type (\ref{eq:IRDopt}) and wish to find an optimal solution $u^h_{rl}$. For notational convenience we omit the subscript $._{rl}$ indicating the current resolution level. In order to solve (\ref{eq:IRDopt}) we use the framework of Sequential Quadratic Programming (SQP) \cite{NW}. Due to the absence of inequality constraints, the SQP method may be interpreted as a kind of Newton method applied to the optimality conditions (\ref{eq:IRDopt_cond}) derived from (\ref{eq:IRDopt}) (cf. \cite{BCN}). 

Let $u_0^h\in\mathbb{R}^n$ be an initial guess for the displacement, typically being the prolongation from the previous resolution level or set to zero for the coarsest space,  and $p^h_0\in\mathbb{R}^m$ be an initial guess for the Lagrange multiplier. The SQP method consist of iteratively updating the approximate displacement $u^h_k$ by an appropriate correction vector $\delta u^h_k$ in the primal space $\mathbb{R}^n$ and $p^h_k$ by a vector $\delta p^h_k$ in the dual space $\mathbb{R}^m$. This can be obtained by solving the KKT system
\begin{equation}\label{eq:KKT}
 \mathcal{K}_k \delta_k :=
  \begin{pmatrix}
    W_k & B_k^\top\\
    B_k & 
  \end{pmatrix}
  \begin{pmatrix}
    \delta u^h_{k}\\
    \delta p^h_{k}
  \end{pmatrix} = 
  -\begin{pmatrix}
    A^h u^h_k + B_k^\top p^h_k\\
    c^h(u^h_k)
  \end{pmatrix} =: -b_k,
\end{equation}
for one level of the SQP algorithm, where $\mathbb{R}^{m\times n}\ni B_k := {\bf B}^h(u^h_k)$ and $W_k\in\mathbb{R}^{n\times n}$ is an approximation of the Lagrangian Hessian
\begin{equation*}
  W_{k,ij}\approx A^h_{ij} + \sum_{l=1}^m \frac{\partial^2 c^h_i}{\partial u_i \partial u_j}p^h_{k,l}.
\end{equation*}
The exact Hessian requires second-order derivatives of the constraints which are expensive to compute. Therefore, we omit this term in our algorithm, i.e., we simply set $W_k = A^h$. 

\begin{Rem}
Observe that whenever $c^h_i$ is approximately an affine function or $p^h_k$ is very small, this approximation is almost accurate. On the other hand, by omitting this term we loose second-order information which may affect the rate of convergence. Nevertheless, in our experience, computing the second-order derivatives of $c^h$ is more costly than simply requiring more iterations to obtain a desired tolerance with $W_k=A^h$.
\end{Rem}

At the lower resolution level, the KKT system \eqref{eq:KKT} may be solved by a direct method, for example, the LU factorization. Unfortunately, when $m,n \gg1$ direct solvers call for a large amount of memory that may be not realizable on a common computer. We therefore allow for an approximate solutions of (\ref{eq:KKT}) obtained by a preconditioned generalized residual method (cf.~Section~\ref{sec:KKT_solver}). Additionally, we incorporate a specific stopping criteria to compromise the computational cost and convergence \cite[Algorithm: Inexact SQP with Gmres and Smart Tests]{BCN}. 

After an approximated solution $\delta_k = (\delta u^h_k,\delta p^h_k)$ is obtained, we use a globalization strategy given by a  line-search, where the next iterate is obtained by updating the previous one with an appropriate scaling of the computed corrections, i.e.,
\begin{equation*}
  \begin{pmatrix}
    u^h_{k+1}\\
    p^h_{k+1}
  \end{pmatrix} = 
  \begin{pmatrix}
    u^h_k\\
    p^h_k
  \end{pmatrix} + \tau_k
  \begin{pmatrix}
    \delta u^h_{k}\\
    \delta p^h_{k} 
  \end{pmatrix},
\end{equation*}
where $\tau_k$ is determined by the backtracking Armijo algorithm satisfying the condition
\begin{equation} \label{eq:Armijo}
  g(u^h_k+\tau_k\delta u^h_{k}) \leq g(u^h_k) - \eta\tau_k\nabla g(u^h_k)\cdot \delta u^h_{k}, \qquad 0<\eta<1,
\end{equation}
where $g$ is an appropriately chosen SQP merit function which plays an extremely important role in the optimization process. To achieve convergence for convex problems one may simply use the merit function $g(u^h_k) = \Vert b_k \Vert ^2$, where $b_k$ is defined in (\ref{eq:KKT}). Nonconvex problems, however, require more complex merit functions to guarantee convergence. Here, we choose the augmented $l_2$ version
\begin{equation}\label{eq:mf}
  g(u^h,\sigma) = S^h(u^h) + \langle p^h,c^h(u^h)\rangle + (\sigma/2) \Vert c^h(u^h) \Vert_2^2,
\end{equation}
where $\sigma > 0$ is an appropriately chosen penalty parameter \cite{NW}.

\begin{Rem}
In order to avoid foldings we extend the Armijo algorithm by a `diffeomorphic test' \cite{JM}. We have experienced that for irregular images the approximated correction vectors may cause the solution to be nondiffeomorphic, i.e., foldings may appear in the deformed grid. This happens whenever the approximated Jacobian determinant changes sign. 
For this reason, we additionally require
\begin{equation}\label{eq:diffeom}
  \min\nolimits_{x\in\Omega_c} V^h(u^h)(x) \geq \delta,
\end{equation}
for some fixed $\delta>0$, where $V^h$ is an approximation of the determinant defined in \eqref{eq:det_h}. More precisely, we choose $\tau_k$ sufficiently small so as to satisfy both \eqref{eq:Armijo} and \eqref{eq:diffeom}.
\end{Rem}

%

Finally, we provide the general scheme for solving (\ref{eq:IRDopt}).

\begin{alg}[One level image registration]\text{ }\label{alg:olir}\\
Given: $u_0^h$, $p_0^h$, stopping criteria\\
Initialize: set $k=0$\\
While the stopping criteria is not satisfied:
\begin{enumerate}
  \item set up the KKT system (\ref{eq:KKT}) for $(u^h_k,p^h_k)$,
  \item compute an approximate solution $(\delta u^h_{k},\delta p^h_{k})$ of (\ref{eq:KKT}),
  \item apply the backtracking Armijo line search to obtain $\alpha_k$ satisfying (\ref{eq:Armijo}),
  \item update $(u^h_{k+1},p^h_{k+1}) = (u^h_k,p^h_k)+\alpha_k (\delta u^h_k,\delta p^h_{k})$,
  \item update $k = k+1$,
\end{enumerate}
\end{alg}

\subsection{KKT solver}\label{sec:KKT_solver}
Our optimization method described in the previous section is based on the assumption that the direct solution of the KKT system (\ref{eq:KKT}) is impossible or hard to obtain. If we are dealing with high resolution images (particularly in 3D) our problem may easily reach a few million unknowns. Therefore, we apply the generalized minimal residual method (GMRES) which is both `relatively cheap' and stable. As known, the GMRES is guaranteed to be convergent in at most $N$ iterations where $N$ is the matrix size. Naturally performing $N$ iterations is not cheap in the sense of the required computational time. Nevertheless, we expect to obtain a reasonable approximation after a relatively small number of iterations. 
Furthermore, since our KKT system is ill-conditioned, we cannot expect fast convergence. Hence, there is a need for a good preconditioner. 
In \cite{EPM} the authors tested different preconditioning methods and concluded that the best result is obtained by applying the block triangular preconditioner
\begin{equation}
 \mathcal{P} = \begin{pmatrix}
                 A &  \\
                 B & -C
               \end{pmatrix}, \quad \text{ where } \quad C = BA^{-1}B^\top - \text{Schur complement matrix}.
\end{equation}
Note, that applying the inverse of $\mathcal{P}$ from the left we get a preconditioned matrix
\begin{equation*}
 \mathcal{K}_{\mathcal{P}} \approx \begin{pmatrix}
                                       I & A^{-1}B^\top\\
                                         & I
                                     \end{pmatrix}
\end{equation*}
which is upper triangular. Thus, one sweep of Gaussian elimination leads to the required solution. Unfortunately, obtaining the exact inverse of $\mathcal{P}$ is almost as costly as inverting the original matrix $\mathcal{K}$ itself. This necessitates the need for a a cheap approximate inverse of $A$ and $C$. 

\subsubsection{Approximate inverse of A} \label{sec:aiA}
Since the matrix $A$ corresponds to the elliptic Navier--Lam\'e operator, one can show that a multigrid solver provides an excellent convergence rate, providing the {\em Poisson ratio} $\nu>0$ is relatively small. For $\nu \rightarrow 0.5$ the rank deficient $\nabla^h \nabla^h\cdot$ operator becomes dominating and $A$ becomes nearly singular \cite{ZSTB2010}. In particular, the solution to the linear elastic system may be perturbed by any divergence-free displacement without affecting the residual. These perturbations may have highly oscillatory behavior and thus the produced errors cannot be smoothed efficiently by the coarser grid information. Since most of the soft human tissues exhibit near incompressible behavior, we reformulate the linear elasticity problem into an equivalent form that does not suffer from the near singularity of the original equation \cite{GGLO}.

For a given $u^h$ let us introduce a new variable $q^h\in H_c$ to be $q^h = -\lambda \nabla^h\cdot u^h$. We now rewrite the discrete Navier--Lam\'e operator (\ref{eq:A}) as
\begin{equation}\label{eq:Aaug}
 A^h_{aug} = \begin{pmatrix}
                         \mu \nabla^h\times\nabla^h\times -2\mu \nabla^h \nabla^h\cdot & \nabla^h\\
                         \nabla^h                                     \cdot & \lambda^{-1}
                       \end{pmatrix}.
\end{equation}
For later use we define a distribution matrix
\begin{equation}
  M^h = \begin{pmatrix}
                    I          & -\nabla^h\\
                    \mu \nabla^h\cdot  & -2\mu \nabla^h\cdot\nabla^h
                  \end{pmatrix}.
\end{equation}
Using the properties of our discrete operators we can easily observe that
\begin{equation}
  A^hM^h = \begin{pmatrix}
                                 \mu(\nabla^h\times\nabla^h \times - \nabla^h \nabla^h\cdot)  & \\
                                 (1+\frac{\mu}{\lambda})\nabla^h\cdot          & -(1+\frac{2\mu}{\lambda})\nabla^h\cdot\nabla^h
                               \end{pmatrix}
\end{equation}
which is a block triangular matrix with Laplace-like operators on the main diagonal. For such matrices a standard Gauss--Seidel relaxation method provides perfect smoothing properties \cite{GGLO}. Here, we used linear interpolation as a prolongation operator and restriction operators constructed with the following stencils
\begin{equation*}
 S_{u_1} = \frac18\begin{pmatrix}
                  1  &     & 1\\
                  2  & \triangledown & 2\\
                  1 &  & 1
                  \end{pmatrix}, \quad
 S_{u_2} = \frac18\begin{pmatrix}
                  1  &  2   & 1\\
                     & \triangleright & \\
                  1 &  2 & 1
                  \end{pmatrix}, \quad
 S_{p}   = \frac14 \begin{pmatrix}
                     1  &  &  1\\
                        &\bullet & \\
                     1 &  & 1
                   \end{pmatrix}.
\end{equation*}
We experienced that such a multigrid solver with a $V(1,1)$ cycle is mesh independent. Furthermore it does not depend on the elasticity parameters, thereby allowing for a nearly incompressible setting.

\subsubsection{Approximate inverse of S} \label{sec:aiS}
For the efficient solver for the Schur complement system with the matrix $C = BA^{-1}B^\top$ we follow the strategy based on approximate commutators \cite{BGL}. This approach consists of finding an approximate solution to the matrix equation
\begin{equation*}
  B^\top X = AB^\top. 
\end{equation*}
Under the assumption that $B$ has full row rank, the least square solution to the above system is
\begin{equation*}
  X \approx (BB^\top)^{-1}BAB^\top.
\end{equation*}
Then we can approximate the inverse of the Schur complement matrix by
\begin{equation}
  C^{-1} \approx (BB^\top)^{-1}BAB^\top(BB^\top)^{-1}.
\end{equation}
By definition the matrix $BB^\top$ is positive semidefinite and strictly positive definite whenever $B$ has full row rank. Unfortunately, $BB^\top$ is typically ill-conditioned and thus iterative methods devoted to symmetric positive definite systems, converge slowly. As suggested in \cite{EPM} we use the complete LU factorization. Note, that the size of this matrix is $m \approx \frac13 N$ and is thus much cheaper than the LU factorization of the original problem. Having the $LU$ form of $BB^\top$, the application of the inverse Schur complement becomes straight forward.

\subsection{Multilevel approach}\label{sec:mrs}
Due to the fact that our constraints $c^h_i$, $1 = 1,...,m$ are nonconvex functions, the problem (\ref{eq:IRDopt_cond}) may possess many different local solutions. Although we cannot guarantee that our algorithm converges to the global minimizer, we can decrease the risk of ending up with a stationary point being far away from it. This is obtained by the multilevel strategy seen in Figure~\ref{fig:multiresolution}. 
\begin{figure}[htb]\centering
\includegraphics[scale=0.5]{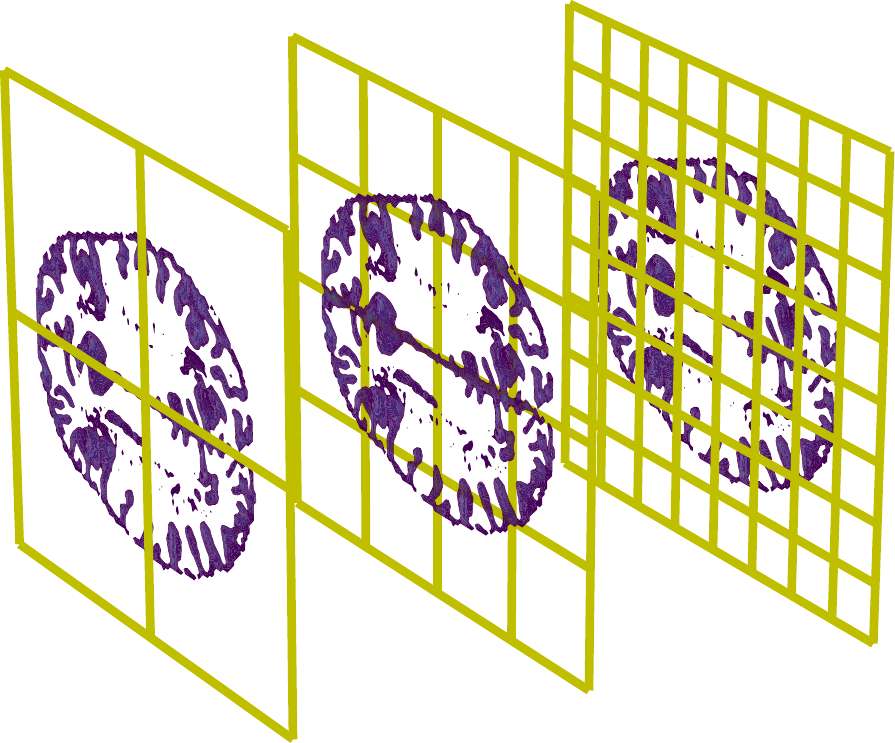}
\caption{Multiresolution coarsening, from left to right}
\label{fig:multiresolution}
\end{figure}
This approach consist of solving a sequence of subproblems with coarser images. These are obtained from the original images by applying a smoothing filter. Starting with the coarsest images we apply the image registration algorithm and then interpolate the obtained solution on the finer domain. One easily observes that the smoother the images are the more regular the problem is and thus providing a higher chance in finding the global solution. Moreover, this approach significantly accelerates and stabilizes the solution algorithm as it provides an excellent initial guess for the current resolution problem.

Here we apply a standard filter based on the stencil $S_p$ described in Section~\ref{sec:aiA}. Let $I_{low}$ denote the restriction operator to a lower resolution level. Then the overall image registration algorithm is expressed as follows.

\begin{alg}[]\text{ }\\
Given: original images $R,T$, integer $rl_{min}$\\
Initialize: $k = rl_{min}$, $u^h_0 = 0$, $p^h_0 = 0$
\begin{enumerate}
  \item produce lower resolution images $R_k = I_{low}^k R$, $T_k = I_{low}^k T$
  \item set up the image registration problem on the resolution level $-k$
  \item apply Algorithm \ref{alg:olir} to get $(u^h_k,p^h_k)$
  \item prolongate $(u^h_k,p^h_k) \rightarrow (u^h_{k-1},p^h_{k-1})$
  \item update $k = k-1$
  \item if $k \geq 0$ go to step 1
\end{enumerate}
\end{alg}
\section{Numerical examples}
Next, we present several examples to demonstrate the advantages of our approach. Among the performed tests we consider a tailored example (Ex1) and a real world example (Ex2). The first case is to show how our algorithm works with an 'ideal' input data, while the second example contains perturbed data requiring some preprocessing. 

In order to compare different results we provide two quantitative information, namely the value ${\sf Elas} := S^h(u^h)$ which approximates the elastic strain energy and ${\sf DMP} = \Vert c^h(u^h)\Vert_{\infty}$ which locally measures the distance of the determined deformed image to the reference one.

\subsection{Tailored example (Ex1)}
Here we demonstrate that our algorithm recovers the true deformation. This example is taken from \cite{FRO}. Let us define a real function
\begin{equation*}
  q(z) := \left(-\frac{1}{8\pi}z^2 + \frac{1}{256\pi^3} + \frac{1}{32\pi}\right)\cos(8\pi z) + \frac{1}{32\pi^2}z\sin(8\pi z).
\end{equation*}
Now let $\Omega = [-0.5,0.5]\times[-0.5,0.5]$ and set $\rho_T$ to be uniformly one, i.e., $\rho_T = 1$ on $\Omega$. Let us define a deformation $\phi_e\colon \Omega\rightarrow\Omega$ as
\begin{equation*}
  \phi_e(x) = \begin{pmatrix} x_1 + 4\,q'(x_1)q(x_2) \\
  							 x_2 + 4\,q(x_1)q'(x_2) \end{pmatrix}.
\end{equation*}
We set $\rho_R$ to be
\begin{equation*}
  \rho_R(x) = \det(\nabla\phi_e(x))(\rho_T\circ\phi_e)(x) = \det(\nabla\phi_e(x)).
\end{equation*}
The above setting guarantees that the mass preserving constraint is exactly satisfied for the displacement field
\begin{equation*}
  u_e(x) = \begin{pmatrix}
           4\,q'(x_1)q(x_2)\\
           4\,q(x_1)q'(x_2)
         \end{pmatrix}.
\end{equation*}
On the other hand we know that there may exist other mass preserving solutions with possibly lower elastic strain energy. Therefore we cannot expect that our algorithm produces approximate solutions that converge to $u_e$. This solution determines an upper bound for the value ${\sf Elas}$, namely ${\sf Elas}_{\max} = S^h(u^h_e)$, where $u^h_e$ stands for the discretization of $u_e$, and a lower bound for the dissimilarity measure ${\sf DMP}_{\min} = 0$. We expect to find an approximate displacement with ${\sf Elas} \leq {\sf Elas}_{\max}$ and ${\sf DMP} \geq 0$.

In Figure~\ref{fig:sin_uniform_input}, we present the input data together with the known deformation field and the corresponding difference image. In this example we discretize $\Omega$ over a grid of $64\times 64$ cells.
\begin{figure}[h]\centering
\includegraphics[scale=0.6]{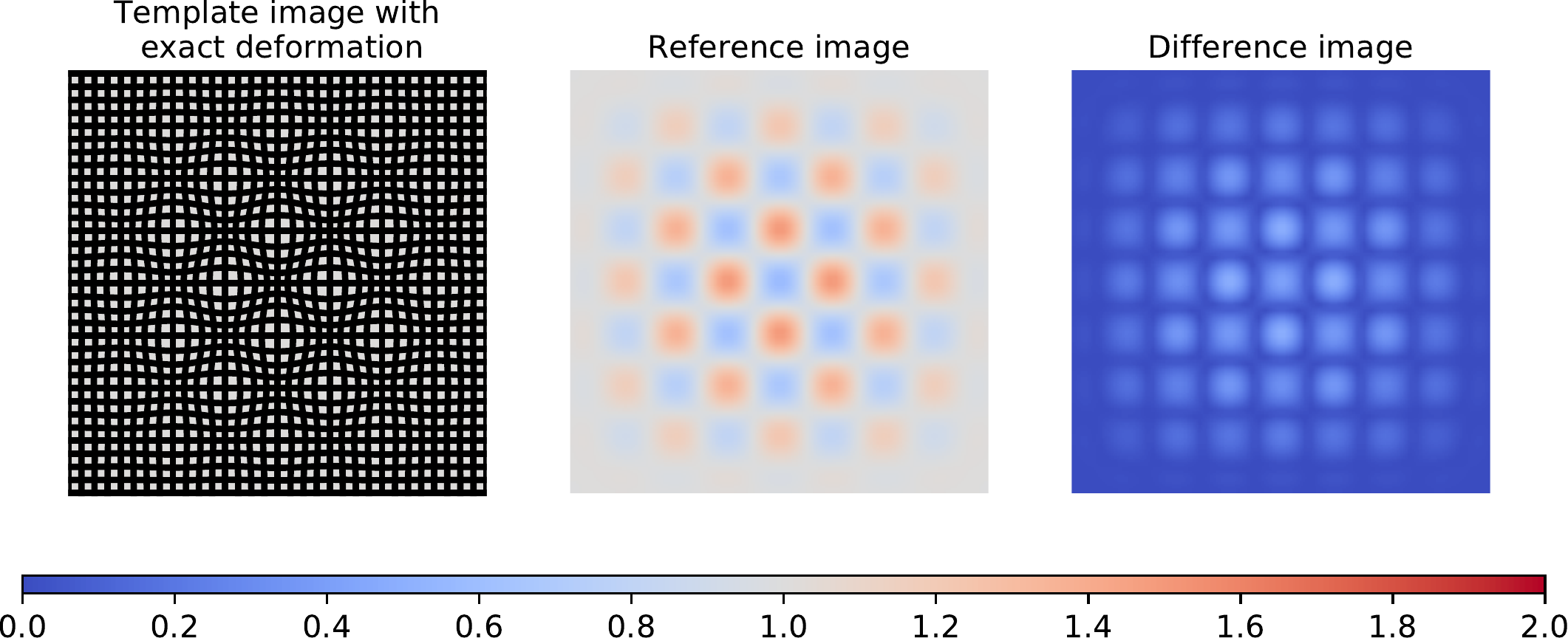}
\caption{Example 1 - input data}
\label{fig:sin_uniform_input}
\end{figure}

Further, we calculate 
\begin{equation*}
  {\sf DE} = \Vert u^h-u^h_e\Vert_{\infty},
\end{equation*}
where $u^h$ stands for the computed approximation, to measure how the minimization of $S^h$ influences the solution.

Figure~\ref{fig:sin_uniform_info} illustrates how the specific properties of the approximated mass preserving solution $u^h_{\text{mp}}$ changes during the optimization process. For comparison purposes, we also depict the results obtained by solving the regularized problem (\ref{op:irp}) with the dissimilarity measure defined in (\ref{op:mp}) for different values of the regularization parameter $\alpha$. We call the approximate solution obtained with this method $u^h_{\alpha}$. In this example we skip the multiresolution strategy and stop the iterative solver for the KKT system (\ref{eq:KKT}) when the residual norm achieves the tolerance level $r_{\text{tol}} = 1\cdot 10^{-5}$. The Lam{\' e} constants are set to be $\mu=1$, $\lambda=0$.
\begin{figure}[htb]\centering 
\hspace*{-3em}\includegraphics[scale=0.40]{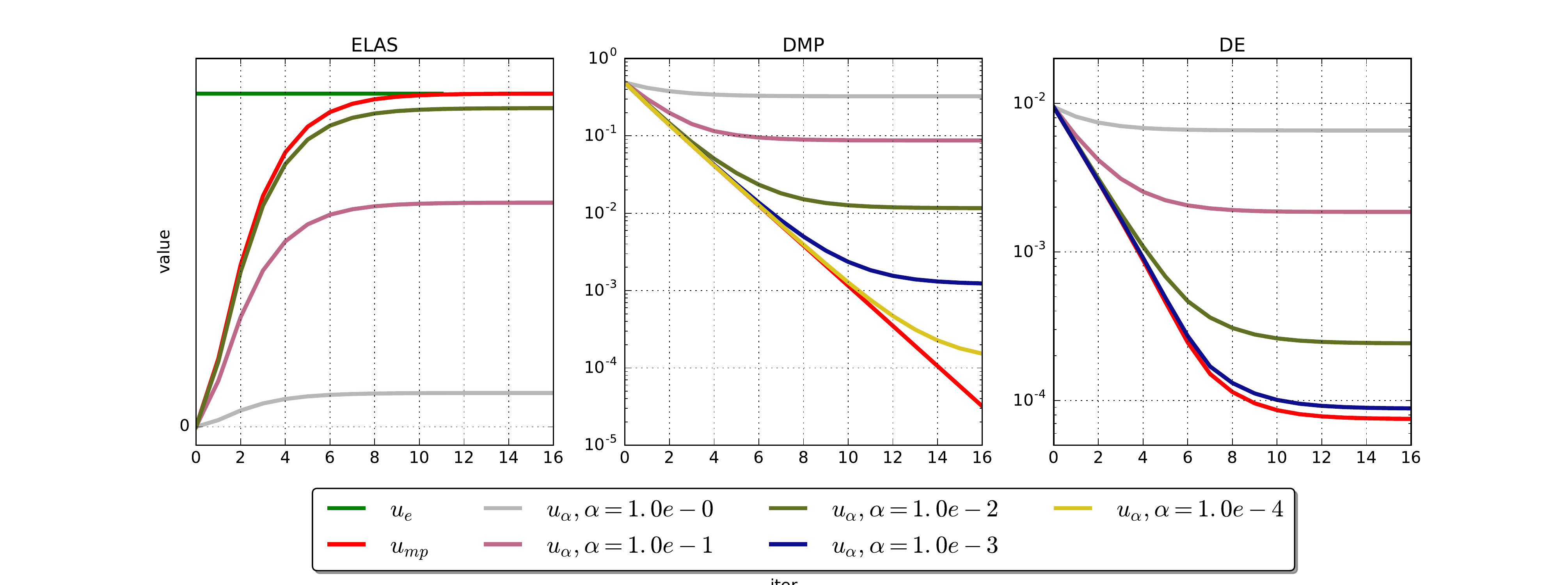}
\caption{Example 1 - {\sf Elas} - {\sf DMP} - {\sf DE}}
\label{fig:sin_uniform_info}
\end{figure}

For $\alpha$ tending to zero, both the approximated displacements $u^h_{\text{mp}}$ and $u^h_{\alpha}$ exhibit similar properties. Indeed, we see that for $\alpha\ll 1$, the curves representing the quantities ${\sf Elas}$, ${\sf DMP}$ and ${\sf DE}$ (Figure \ref{fig:sin_uniform_info} left, middle and right respectively) overlap, which is to be expected. 

The main difficulty in the regularized problem is that we do not know the regularization parameter $\alpha$ a priori. One can naturally use the L-curve or continuation method \cite{CMRS,PCH,CRV} in order to find an optimal parameter, but these approaches are rather computationally expensive. In practice, one simply chooses an artificial value. Figure~\ref{fig:sin_uniform_iter} demonstrate how strongly this choice affects the optimization process and the solution. If $\alpha$ is chosen too large, we obtain very smooth approximations after few iterations, whereupon the process stagnates and the solution stays far away from $u_e$. We obtain a reasonably good approximation for $\alpha<1\cdot 10^{-3}$. In this situation, however, we must solve extremely ill-conditioned systems. Figure~\ref{fig:sin_uniform_iter}(left) presents the required number of iterations for the GMRES solver with algebraic multigrid preconditioner \cite{BeOlSc2008}. In Figure~\ref{fig:sin_uniform_iter}(right) we present how the regularization affects the obtained deformed image model.

\begin{figure}[htb]\centering
\hspace*{-3em}\includegraphics[scale=0.39]{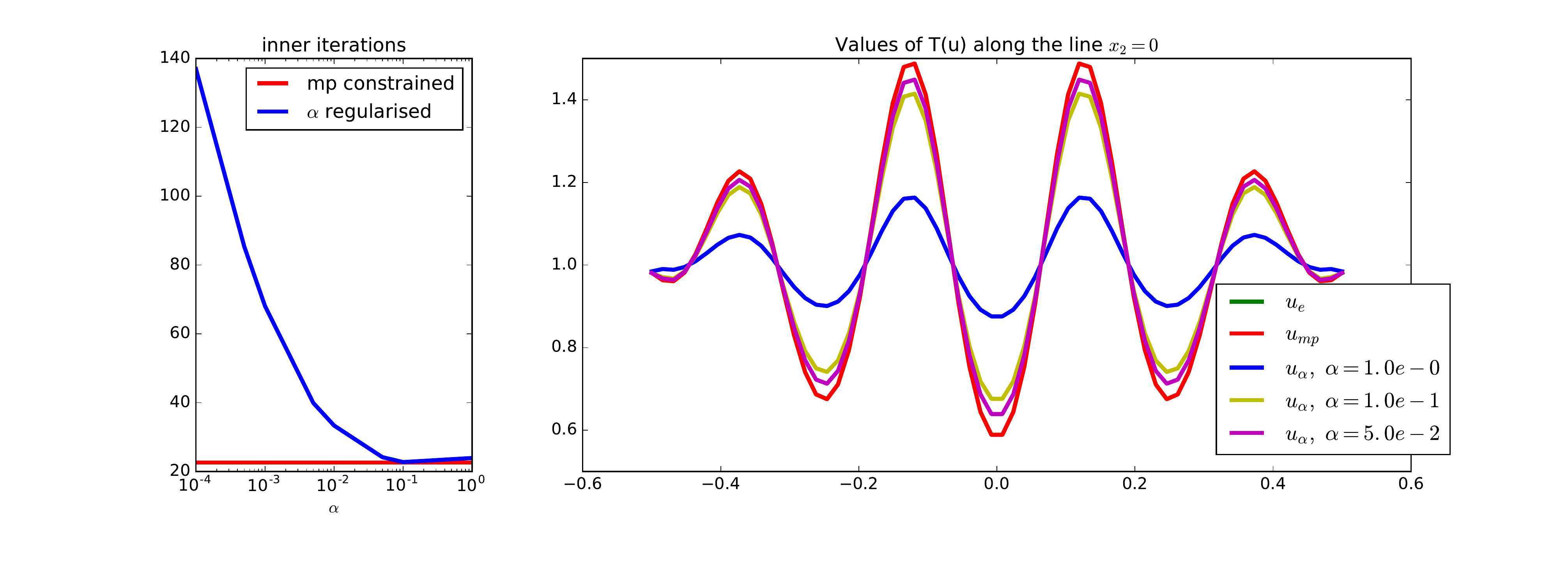}
\caption{Example 1 - left: average number of inner iterations for one KKT solve, right: values of deformed image along the axis $y=0$}
\label{fig:sin_uniform_iter}
\end{figure}

\subsection{Real world example (Ex2)}
In this example we process real world data obtained from \cite{BrainWeb}. They illustrate the cortical tissue of human brain at different positions in $z$ dimension. The image resolution is $128\times 128$. We first apply a Gaussian filter on these images to get rid off the present noise. Then, in order to force the mass preserving condition, we scale these data to get intensity values in range $[\delta,1]$ and modify in such a way to obtain the same total mass (sum of all pixel intensities). The value $\delta=0.03$ prevents instabilities during the optimization process. We present the input data and the corresponding difference image on Figure~\ref{fig:brain_input}. We see that the biggest difference appears on the boundary of the cortex tissue. This situation is very realistic and typically comes from spatial misalignment. 
\begin{figure}[h]\centering
\includegraphics[scale=0.6]{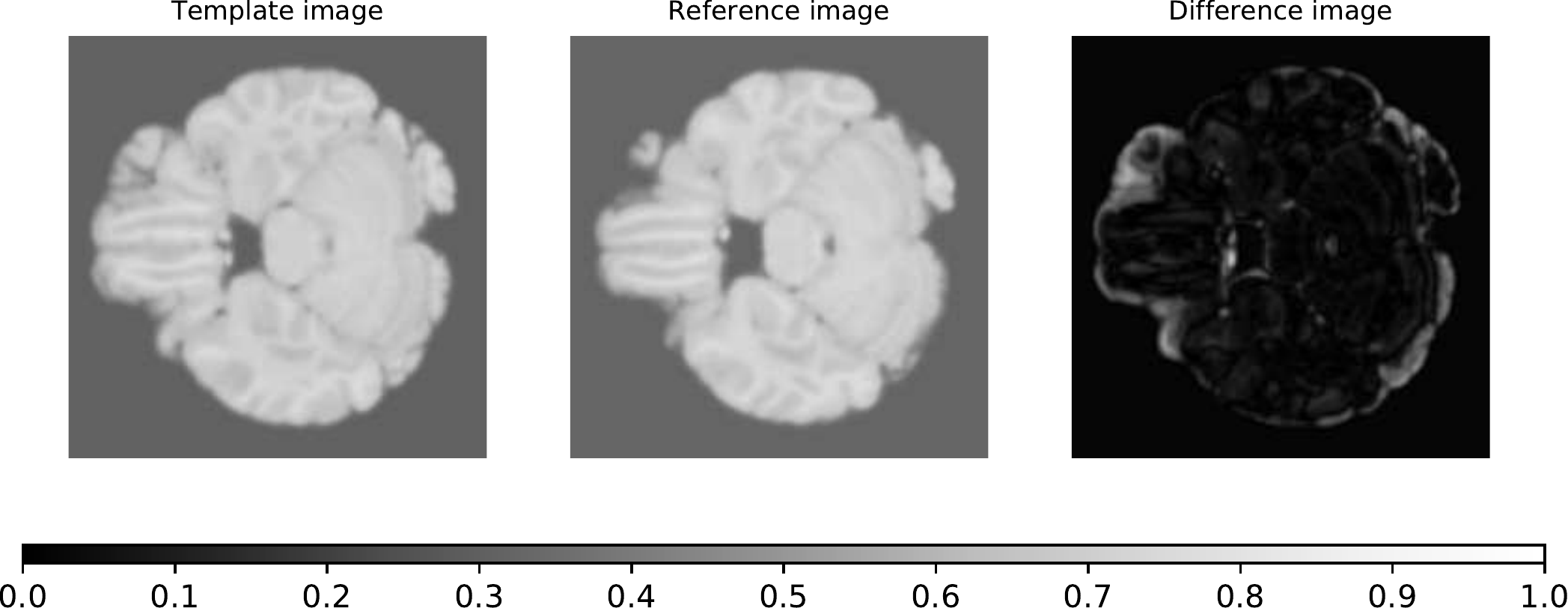}
\caption{Example 2 - input data}
\label{fig:brain_input}
\end{figure}

We adapt the algorithm to work on free resolution levels, i.e., at each resolution level $rl$ we stop the iterations as soon as 
\begin{equation*}
 {\sf DMP}_{rl} < C_{rl} \cdot 10^{-3},
\end{equation*}
where $C_{rl}$ corresponds to the cell size at level $rl$, which is relative to the cell size at the finest resolution, i.e., $C_0=1$, $C_{1} = 4$, $C_{2}=16$. As before we set $\mu=1$, $\lambda = 0$.
\begin{figure}[htb]\centering 
\hspace*{-2em}\includegraphics[scale=0.39]{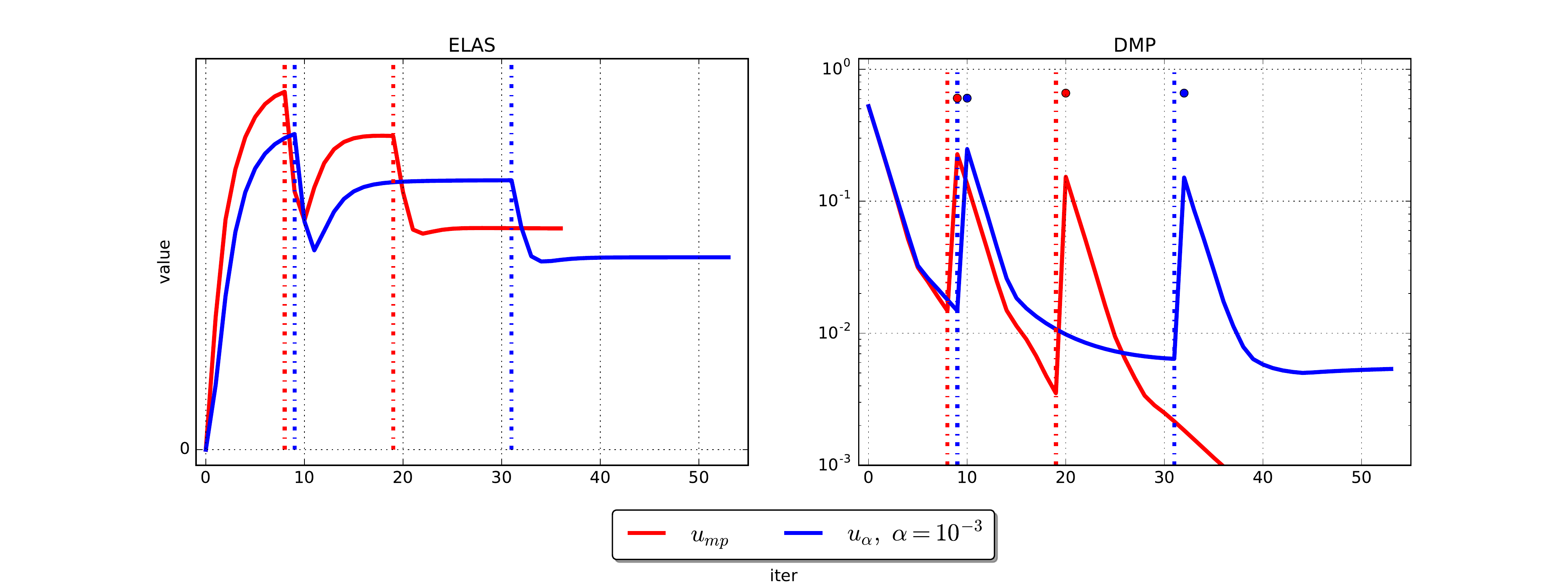}
\caption{Example 2 - {\sf Elas} - {\sf DMP} - {\sf DMP$_{\text{global}}$}, vertical dotted lines separate different resolution levels, $\bullet$ denote the reference value of ${\sf DMP}$ for zero vector}
\label{fig:brain_info}
\end{figure}

We present the results on Figures \ref{fig:brain_info} and \ref{fig:brain_deform}. Information from all resolution levels are posed at one figure and separated by vertical dotted lines. Thus we can see the overall performance of the registration algorithm. We can see that ${\sf DMP}$ is reduced significantly by the solution from the coarser grid and thus justifies the use of the multiresolution strategy. We performed $36$ iterations to obtain a solution with the desired {\sf DMP} tolerance. The regularized problem with parameter $\alpha= 10^{-3}$, however, does not lead to a solution with ${\sf DMP}< 10^{-3}$ and requires in fact a smaller $\alpha$. Figure \ref{fig:brain_deform} shows the deformation grid which allows for an almost perfect match of the images, while the mass is preserved with high accuracy.

\begin{Rem}
We point out that the ratio between the Lam{\' e} constants $\nicefrac{\mu}{\lambda}$ has a big influence on the determined solution. Whenever $\nicefrac{\mu}{\lambda}$ is large the algorithm allows for solutions that are not divergence-free. On the other hand, when $\nicefrac{\mu}{\lambda}$ tends to zero, the divergence part of the elastic strain energy dominates. In this situation displacement fields with a dominating divergence-free part are more likely to be chosen.
\end{Rem}

\begin{figure}[htb]\centering
\includegraphics[scale=0.6]{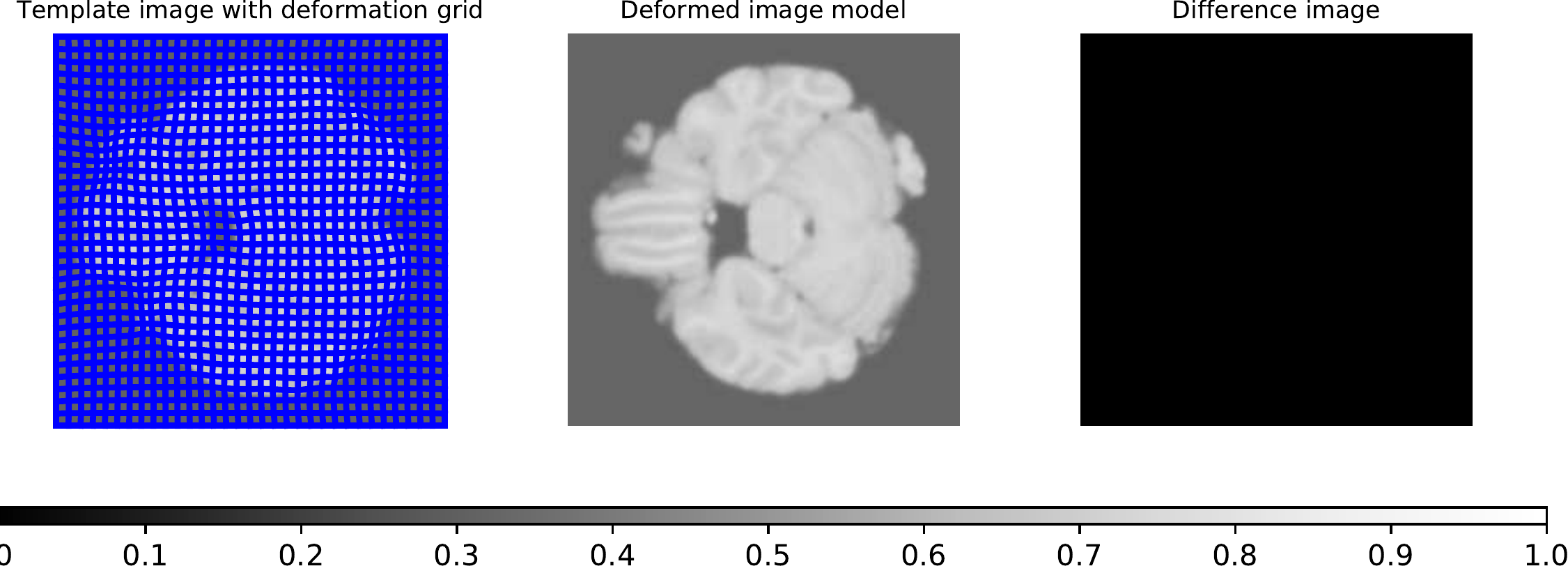}
\caption{Example 2 - deformation grid - deformed image model - difference image}
\label{fig:brain_deform}
\end{figure}

\section{Conclusion}
We considered a 2D image registration problem as an optimization problem with a fully nonlinear partial differential equation as a hard constraint. This constraint expresses the preservation of mass while the cost functional controls the linear elastic strain energy of a deformation. We proved that the problem stated in this form has a weak solution in a Sobolev space under fairly mild assumptions. Furthermore, we developed a numerical scheme to solve the problem and demonstrated its robustness on synthetic and real world data. The method involves no regularization parameter and allows to match the images accurately while preserving the mass with high precision.

\begin{appendix}
\section{Discretization} \label{app:discretization}
\subsection{Staggered grid}\label{subsec:staggered_grid}
We assume that the input data $R$ and $T$ are $n_1\times n_2$ matrices describing the pixels intensities. We identify each pixel with a square, whose side length is $h$ and assume that the given data are associated with the cell centers. The unknown displacements $u$ and the corresponding deformations $\phi$ are discretized as follows. Let $e_i$ be the $i$'th Cartesian axis vector, $i=1,2$. We store the variables $u_i^h$ ($\phi_i^h$ respectively) on the center of grid edges whose orientation is consistent with $e_i$, where the superscript $\cdot^h$ indicates the discrete approximation. For simplicity let us introduce the following meshes
\begin{equation}
\begin{aligned}
 \Omega_n &= \Big\{(ih,jh):\ i\in\{0,...,n_1\},\ j\in\{0,...,n_2\}\Big\}\\
 \Omega_c &= \Big\{((i+\nicefrac12)h,(j+\nicefrac12)h):\ i\in\{0,...,n_1-1\},\ j\in\{0,...,n_2-1\}\Big\}\\
 \Omega_{u_1} &= \Big\{((i+\nicefrac12)h,jh):\ i\in\{0,...,n_1-1\},\ j\in\{0,...,n_2\}\Big\}\\
 \Omega_{u_2} &= \Big\{(ih,(j+\nicefrac12)h):\ i\in\{0,...,n_1\},\ j\in\{0,...,n_2-1\}\Big\}.
\end{aligned}
\end{equation}
Moreover, let $H_n$, $H_c$, $H_{u_1}$, $H_{u_2}$ be appropriate spaces of discrete functions defined on $\Omega_n$, $\Omega_c$, $\Omega_{u_1}$ and $\Omega_{u_2}$ respectively. Additionally we define $H_u$ to be $H_u := H_{u_1}\times H_{u_2}$. With this notation we can write $R,T\in H_c$, $\phi_i^h,u_i^h\in H_{u_i}$, $i=1,2$ and $\phi^h,u^h\in H_u$. It is worth to note that $u^h$ is the optimization variable and the deformation $\phi^h$ appears implicitly as $\phi^h(x) = x + u^h(x)$, $x\in\Omega_c$.

Next, we define discrete derivatives of the quantities defined on the staggered grid $\Omega_{u_1}\times\Omega_{u_2}$ as
\begin{equation}
\begin{aligned}
  D_1^i u^h(x) &= \frac1h \left(u_i^h(x_1+\nicefrac{h}{2},x_2) - u_i^h(x_1-\nicefrac{h}{2},x_2)\right),\quad x\in \Omega_n \text{ if }i=1 \text{ or }x\in \Omega_c \text{ if } i = 2,\\
  D_2^i u^h(x) &= \frac1h \left(u_i^h(x_1,x_2+\nicefrac{h}{2}) - u_i^h(x_1,x_2-\nicefrac{h}{2})\right),\quad x\in \Omega_c \text{ if }i=1 \text{ or }x\in \Omega_n \text{ if } i = 2,
\end{aligned}
\end{equation}
where near the boundary we utilize homogeneous Dirichlet boundary conditions. Observe that with the above second-order accurate definition, the first-order derivatives of a certain quantity in $H_u$ are not collocated with itself. Indeed, one easily notices that $D_i^i\colon H_u\rightarrow H_n$, but $D_i^j\colon H_u\rightarrow H_c$ for $i\neq j$. Moreover, locations of the normal and tangential derivatives differ. Such collocation of unknowns and their derivatives calls for introducing projection operators when the constraints are discretized. Therefore we define operators $P_{u_i\rightarrow c}\colon H_{u_i}\rightarrow H_c$, $P_{u_i\rightarrow n}\colon H_{u_i}\rightarrow H_n$, $i=1,2$
as bilinear interpolation, e.g.
\begin{equation}
  P_{u_1\rightarrow c}u^h(x) = \frac12\left( u^h_1(x_1-\nicefrac{h}{2},x_2)+u^h_1(x_1+\nicefrac{h}{2},x_2) \right), \qquad x\in\Omega_c.
\end{equation}

\subsection{Image model} \label{sec:im}
An important issue is the observation that the input data may have extremely irregular structure. This would lead to instabilities when a numerical method is implemented. In order to minimize the risk of failure there is a need for additional regularization. As one of the building blocks we require a continuous and differentiable representation of the image data $T$. The simplest and computationally cheapest choice would be the linear interpolation method. Clearly the differentiability assumption is then not satisfied. Therefore, following \cite{HM4} we use the cubic B-spline interpolation scheme \cite{MS}. Due to the Curry--Schoenberg theorem \cite{SCHOENBERG19731} we can find a unique B-spline representation of any order for our discrete data $T$. The order 3 is chosen due to the minimum curvature property. With this choice we have the following image model
\begin{equation} \label{eq:BsImage}
  \rho_{T}(x) = \mathcal{I}_{bs}[T](x) := \sum_{i=0}^3\sum_{j=0}^3 B_i(\bar{x}_1)B_j(\bar{x}_2)\gamma_{\bar{i}+i,\bar{j}+j},
\end{equation}
where $(\bar{i},\bar{j})$ is the index related to the control cell containing $x = (x_1,x_2)$ by the formula $\bar{i} = \left\lfloor \nicefrac{x_1}{h}\right\rfloor-1$, $\bar{j} = \left\lfloor \nicefrac{x_2}{h}\right\rfloor-1$ and $\bar{x}$ is the relative position of $x$ inside that cell, i.e., 
\[
 \bar{x}_1 = \nicefrac{x_i}{h} - (\bar{i}+1),\qquad \bar{x}_2 = \nicefrac{x_2}{h} - (\bar{j}+1).
\]
The functions $B_i,\ i=0,...,3$ are the Bernstein basis polynomials of degree $3$.
The coefficients $\gamma_{i_1,i_2},\ i_k = -1,...,n_k+1$ may be precomputed. 
Observe that the continuous function $\rho_T$ defined in this way possesses smooth derivatives up to third order which can be calculated directly in a simple and computationally efficient way. 
The B-splines method gives a very good compromise between the accuracy of the interpolation and the reasonable computation time. 

Let $u^h\in H_u$ be a given displacement function and $\phi^h=id +u^h$ be its corresponding deformation. With this model at hand we can easily calculate the values of $\rho_T\circ \phi^h$ and its derivatives at the shifted cell centers
\begin{equation}\label{eq:yy}
  (y_1,y_2) = (x_1+P_{u_1\rightarrow c}u_1^h,x_2+P_{u_2\rightarrow c}u_2^h),\quad (x_1,x_2)\in\Omega_c. 
\end{equation}
Let $\rho_T\circ \phi^h\in H_c$ denote the warped template image uniquely determined by the displacement $u^h$ and its derivative $(\nabla^h\rho_T)\circ \phi^h\in H_c\times H_c$, i.e.,
\begin{align*}
  \rho_T\circ \phi^h(x) &= \rho_T(y),\qquad x\in\Omega_c\\
  (\nabla^h_1\rho_T)\circ \phi^h(x) &= \partial_{x_1}\rho_T(y), \qquad x\in\Omega_c \\
  (\nabla^h_2\rho_T)\circ \phi^h(x) &= \partial_{x_2}\rho_T(y),\qquad x\in\Omega_c,
\end{align*}
where $y = (y_1,y_2)$ is given as in (\ref{eq:yy}). Then the linear operator $\delta(\rho_T\circ \phi^h)\in \mathcal{L}(H_u,H_c)$ being the Fr{\' e}chet derivative of $\rho_T\circ \phi^h$ is
\begin{equation}
\begin{aligned}
  \delta(\rho_T\circ \phi^h) [\cdot] = (\nabla^h_1\rho_T\circ \phi^h)\,P_{u_1\rightarrow c}[\cdot] + (\nabla^h_2\rho_T\circ \phi^h)\,P_{u_2\rightarrow c}[\cdot].
\end{aligned}
\end{equation}

\subsection{Constraints} \label{sec:dc}
Observe that the constraint $c$ contains the Jacobian dterminant $\det(\nabla\phi)$, which couples all elements of the Jacobian matrix of $\phi$. This is in fact the most challenging element to discretize. As mentioned in Section \ref{subsec:staggered_grid}, derivatives are located at different locations. A natural solution would be to apply the projection operator. However,  this approach may lead to instabilities as observed in \cite{LR}. Therefore, we follow the method proposed in \cite{HM}. Instead of discretizing $\det(\cdot)$ in $\Omega_c$ we calculate the approximated change of volume for each cell. This method is accurate up to second-order and allows to easily detect instabilities \cite{LR}. More precisely, let $x\in\Omega_c$ and $V(x)$ denote the cell centered at point $x=(x_1,x_2)$. Then we have
\begin{equation}  \label{eq:det}
\begin{aligned}
  \det(\nabla\phi)(x) = \frac{1}{h^2}vol(\phi(V(x)) + \mathcal{O}(h^2) = \frac{1}{h^2}\int_{\phi(V(x))} dx + \mathcal{O}(h^2)
\end{aligned}
\end{equation}
and the latter is approximated by a procedure precisely described in \cite{HM}.

Now we define the mapping $\det^h\colon H_u\to H_c$ by
\begin{equation}\label{eq:det_h}
  V^h(u^h)(x) := \frac{1}{h^2}vol(\phi^h(V(x))), 
  \quad x\in\Omega_c,
\end{equation}
with $\phi^h = id + u^h$. From the above discussion, this yields a consistent second-order approximation of the Jacobian determinant \cite{LR}. Correspondingly, let $\delta V^h(u^h)\in \mathcal{L}(H_u,H_c)$ denote the Fr{\' e}chet derivative of $V^h$ at $u^h$. Similarly, we consider the nonlinear constraint $c^h\colon H_u\rightarrow H_c$ where
\begin{equation} \label{eq:c}
  c^h(u^h)(x) = V^h(u^h)(x)(\rho_T\circ \phi^h)(x) - \rho_R(x), \qquad x\in\Omega_c.
\end{equation}
Furthermore, its Frech\'et derivative at $u^h$ denoted by $B^h(u^h) \in \mathcal{L}(H_u,H_c)$ is given by
\begin{equation}\label{eq:dc}
  \delta c^h(u^h)[\cdot](x) = \delta V^h(u^h)[\cdot](x)(\rho_T\circ \phi^h)(x) + \delta(\rho_T\circ \phi^h)[\cdot](x)\,V^h(u^h)(x),\qquad x\in\Omega_c.
\end{equation}

\subsection{Objective functional} \label{sec:dr}
Note,  that the Fr{\' e}chet derivative of $\mathcal{S}$ in $H^1(\Omega;\R^2)$ is
\begin{equation*}
\begin{aligned}
  d\mathcal{S}(u)[v] &=  \langle\mu\Delta u + (\mu+\lambda)\nabla(\nabla\cdot u), v\rangle_{H^{-1},H^1} =: \langle\mathcal{A}(u),v\rangle_{H^{-1},H^1}\quad \text{ for all } v\in H^1(\Omega;\R^2),
\end{aligned}
\end{equation*}
where $\Delta$ is the Laplace operator and $\nabla\cdot$ denotes divergence operator. A straightforward approach would be to discretize the operator $\mathcal{A}$ by $A^h$ and define $S^h(u^h)$ to be $S^h(u^h) := \frac12 \langle u^{h},A^hu^h\rangle$ as in \cite{HW}. This approach, however, has a significant drawback which causes the multigrid method to be inefficient for $\lambda \gg \mu$. Therefore, we choose a mimetic discretization (cf.~\cite{GGLO}) in order to preserve the properties of the continuous operator. Let $\nabla\times$ denote the self-adjoint curl operator. Then for any continuous vector field $u$ and scalar field $f$ we have the identities:
\begin{enumerate}
  \item[(i)] $\Delta u = \nabla(\nabla\cdot u) - \nabla\times(\nabla\times u)$
  \item[(ii)] $\nabla\times(\nabla f) = 0$
  \item[(iii)] $\nabla\cdot(\nabla\times u) = 0$
\end{enumerate}
Using the property (i) we can write
\begin{equation*}
\begin{aligned}
  d\mathcal{S}(u)[v] &=  \langle\mu\nabla\times(\nabla\times u) - (\lambda+2\mu)\nabla(\nabla\cdot u), v\rangle_{H^{-1},H^1}  =: \langle\mathcal{G}^*\mathcal{G}u, v\rangle_{H^{-1},H^1},
\end{aligned}
\end{equation*}
where $\cdot^*$ denotes the adjoint operator and
\begin{equation*}
  \mathcal{G} = \begin{pmatrix} \sqrt{\mu} & 0\\ 0 & \sqrt{2\mu + \lambda} \end{pmatrix}\begin{pmatrix} \nabla\times  \\ \nabla\cdot \end{pmatrix}.
\end{equation*}
Following \cite{GGLO} we assemble operators $\nabla^h\times$ and $\nabla^h\cdot$ satisfying properties (ii) and (iii) on the discrete spaces, and set
\begin{equation} \label{eq:A}
  G^h := \begin{pmatrix} \sqrt{\mu} & 0\\ 0 & \sqrt{2\mu + \lambda} \end{pmatrix}\begin{pmatrix} \nabla^h\times \\ \nabla^h\cdot \end{pmatrix}, \qquad A^h := h^2G^{h,\top}G^h.
\end{equation}
Finally, our mimetic discretization of the objective functional reads
\begin{equation} \label{eq:S}
  S^h(u^h) = \frac12 \langle u^h, A^hu^h\rangle,
\end{equation}
whose Fr{\'e}chet derivative is given by
\begin{equation} \label{eq:dS}
  \delta S^h(u^h)[v^h] = \langle u^h,A^h v^h\rangle\qquad \forall v^h\in H_u.
\end{equation}
\end{appendix}

\bibliographystyle{plain} 
\bibliography{references} 

\begin{thebibliography}{10}

\bibitem{EPM}
O.~Axelsson and J.~Karatson.
\newblock Efficient preconditioned solution method for elliptic partial
  differential equation.
\newblock {\em Bentham Science}, 2011.

\bibitem{BK}
R.~Bajcsy and S.~Kova{\u c}i{\u c}.
\newblock Multiresolution elastic matching.
\newblock {\em Computer Vision, Graphics and Image Processing}, 46:1--21, 1989.

\bibitem{BeOlSc2008}
W.~N. Bell, L.~N. Olson, and J.~Schroder.
\newblock Pyamg: Algebraic multigrid solvers in python.
\newblock \url{http://www.pyamg.org}, 2013.
\newblock Version 2.1.

\bibitem{BGL}
M.~Benzi, G.H. Golub, and J.~Liesen.
\newblock Numerical solutions of saddle point problems.
\newblock {\em Acta Numerica}, 14:1--137, May 2005.

\bibitem{burger2013hyperelastic}
M.~Burger, J.~Modersitzki, and L.~Ruthotto.
\newblock A hyperelastic regularization energy for image registration.
\newblock {\em SIAM Journal on Scientific Computing}, 35(1):B132--B148, 2013.

\bibitem{BCN}
R.~Byrd, F.~Curtis, and J.~Nocedal.
\newblock An inexact sqp method for equality constrained optimization.
\newblock {\em SIAM Journal of Optimization}, 19(1):351--369, 2008.

\bibitem{CMRS}
D.~Calvetti, S.~Morigi, L.Reichel, and F.~Sgallari.
\newblock Tikhonov regularization and l-curve for large discrete ill-posed
  problems.
\newblock {\em Journal of Computational and Applied Mathematics},
  123(1--2):423--446, 2000.

\bibitem{CRM}
G.E. Christensen, R.~Rabbit, and M.I. Miller.
\newblock Deformable templates using large deformations kinematics.
\newblock {\em IEEE Transactions on Medical Images}, 5(10), 1996.

\bibitem{CC2009}
N.~Chumchob and K.~Chen.
\newblock A robust affine image registration method.
\newblock {\em International Journal of Numerical Analysis and Modeling},
  6(2):311--334, 2009.

\bibitem{MEC}
P.G. Ciarlet.
\newblock {\em Mathematical Elasticity}, chapter 1: Three Dimensional
  elasticity.
\newblock Elsevier Science Publishers B. V., 1988.

\bibitem{DAC}
B.~Dacorogna.
\newblock {\em Direct methods in the calculus of variations}.
\newblock Applied Mathematical Sciences. Springer Berlin Heidelberg, New York,
  2 edition, 2008.

\bibitem{DM}
B.~Dacorogna and J.~Moser.
\newblock On partial differential equations involving the jacobian determinant.
\newblock {\em Annales de l'l. Henri Poincar{\' e} Analyse non lin{\' e}aire},
  7(1):1--26, 1990.

\bibitem{CFD}
A.W. Date.
\newblock {\em Introduction to Computational Fluid Dynamics}.
\newblock Cambridge University Press, 2005.

\bibitem{Korn}
L.~Desvillettes and C.~Villani.
\newblock On a variant of korn's inequality arising in statistical mechanics.
\newblock {\em ESAIM: Control, Optimization and Calculus of Variations},
  8:603--619, 2002.

\bibitem{LE}
L.~Evans.
\newblock Partial differential equations and monge-kantorovich mass transfer -
  lecture notes.
\newblock Berkeley, 2001.

\bibitem{FM1}
B.~Fisher and J.~Modersitzki.
\newblock Fast diffusion registration.
\newblock {\em Contemporary Mathematics}, 313:117--129, 2002.

\bibitem{FM2}
B.~Fisher and J.~Modersitzki.
\newblock Curvature based image registration.
\newblock {\em Journal of Mathematical Imaging and Vision}, 18(1):81--85, 2003.

\bibitem{FM3}
B.~Fisher and J.~Modersitzki.
\newblock A unified approach to fast image registration and a new curvature
  based registration technique.
\newblock {\em Linear Algebra and its Applications}, 23(7):107--124, 2004.

\bibitem{FRO}
B.D. Froese.
\newblock A numerical method for the elliptic monge-amp{\` e}re equation with
  transport boundary conditions.
\newblock {\em SIAM Journal of Scientific Computing}, 34(3):1432--1459, 2012.

\bibitem{GGLO}
F.~Gaspar, J.~Gracia, F.~Lisbona, and C.~Oosterlee.
\newblock Distributive smoothers in multigrid for problems with dominating
  grad-div operators.
\newblock {\em Numerical Linear Algebra with Applications}, 15:661--683, 2008.

\bibitem{gigengack2012motion}
F.~Gigengack, L.~Ruthotto, M.~Burger, C.H. Wolters, X.~Jiang, and K.P.
  Schafers.
\newblock Motion correction in dual gated cardiac pet using mass-preserving
  image registration.
\newblock {\em IEEE transactions on medical imaging}, 31(3):698--712, 2012.

\bibitem{MDG}
M.D. Gunzburger.
\newblock {\em Perspectives in flow controll and optimization}.
\newblock SIAM, Philadelphia, 2003.

\bibitem{HHM}
E.~Haber, R.~Horesh, and J.~Modersitzki.
\newblock Numerical optimization for constrained image registration.
\newblock {\em Numerical Linear Algebra with Applications}, 17:343--359, 2010.

\bibitem{HM}
E.~Haber and J.~Modersitzki.
\newblock Numerical methods for volume preserving image registration.
\newblock {\em Inverse Problems}, 20(5), 2004.

\bibitem{HM4}
E.~Haber and J.~Modersitzki.
\newblock A multilevel method for image registration.
\newblock {\em SIAM Journal of Scientific Computing}, 27(5):1594--1607, 2006.

\bibitem{HHH}
J.V. Hajnal, D.L.G. Hill, D.J. Hawkes, and editors.
\newblock {\em Medical Image Registration}.
\newblock CRC Press, Boca Raton, 2001.

\bibitem{PCH}
P.C. Hansen.
\newblock {\em Rank-Deficient and Ill-Posed Problems: Numerical Aspects of
  Linear Inversion}.
\newblock SIAM, Philadelphia, 1997.

\bibitem{HW}
S.~Henn and K.~Witsch.
\newblock Iterative multigrid regularization techniques for image matching.
\newblock {\em SIAM Journal of Scientific Computing}, 23(4):1077--1093, 2001.

\bibitem{BrainWeb}
kkk.
\newblock Brainweb: Simulated brain database.
\newblock \url{http://brainweb.bic.mni.mcgill.ca/brainweb}.
\newblock Accessed: 2016-04-14.

\bibitem{KU2003}
J.~Kybic and M.~Unser.
\newblock Fast parametric elastic image registration.
\newblock {\em IEEE RANSACTIONS ON IMAGE PROCESSING}, 12(11), 2003.

\bibitem{LWMFS}
L.~Lemieux, U.C. Wieshmann, N.F. Moran, D.R. Fish, and S.D. Shorvon.
\newblock The detection and significance of subtle changes in mixed-signal
  brain lesions by serial mri scan matching and spatial normalization.
\newblock {\em Medical Image Analysis}, 2:227--242, 1998.

\bibitem{MCVMS}
F.~Maes, A.~Collignon, D.~Vandermeulen, G.~Marchal, and P.~Suetens.
\newblock Multimodality image registration by maximization of mutual
  information.
\newblock {\em IEEE Transactions on Medical Images}, 16(2), April 1997.

\bibitem{MA}
V.R.S. Mani and S.~Arivazhagan.
\newblock Survey of medical image registration.
\newblock {\em Journal of Biomedical Engineering and Technology}, 1(2):8--25,
  2013.

\bibitem{JM}
J.~Modersitzki.
\newblock {\em Numerical Methods for Image Registration}.
\newblock Oxford University Press, New York, 2004.

\bibitem{JMa}
J.~Modersitzki.
\newblock Flirt with rigidity - image registration with a local non-rigidity
  penalty.
\newblock {\em International Journal of Computer Vision}, 76:153--163, 2008.

\bibitem{NPIRGEN}
Andriy Myronenko, Xubo Song, and Miguel~A. Carreira-Perpinan.
\newblock {Non-parametric Image Registration Using Generalized Elastic Nets}.
\newblock In Xavier Pennec and Sarang Joshi, editors, {\em {1st MICCAI Workshop
  on Mathematical Foundations of Computational Anatomy: Geometrical,
  Statistical and Registration Methods for Modeling Biological Shape
  Variability}}, pages 156--163, Copenhagen, Denmark, 2006.

\bibitem{NW}
J.~Nocedal and S.J. Wright.
\newblock {\em Numerical Optimization}.
\newblock Springer-Verlag, New York, 1999.

\bibitem{PMS}
C.~P{\" o}schl, J.~Modersitzki, and O.~Scherzer.
\newblock A variational setting for volume constrained image registration.
\newblock {\em Inverse Problems and Imaging}, 4(3), 2010.

\bibitem{TR}
T.~Rehman.
\newblock {\em Efficient Numerical Method for Solution of $L^2$ Optimal Mass
  Transprt Problem}.
\newblock PhD thesis, School of Electrical and Computer Engineering, 2010.

\bibitem{VBS96}
K.~Rohr, H.~S. Stiehl, R.~Sprengel, W.~Beil, T.~M. Buzug, J.~Weese, and M.~H.
  Kuhn.
\newblock {\em Point-based elastic registration of medical image data using
  approximating thin-plate splines}, pages 297--306.
\newblock Springer Berlin Heidelberg, Berlin, Heidelberg, 1996.

\bibitem{LRmt}
L.~Ruthotto.
\newblock Mass-preserving registration of medical images.
\newblock Master's thesis, Fachbereich Mathematik und Informatik, Westf{\"
  a}lische Wilhelms-Universit{\" a}t M{\" u}nster, Deutschland, 2010.

\bibitem{LR}
L.~Ruthotto.
\newblock {\em Hyperelastic Image Registration - Theory, Numerical Methods and
  Applications}.
\newblock PhD thesis, Fachbereich Mathematik und Informatik, Westf{\" a}lische
  Wilhelms-Universit{\" a}t M{\" u}nster, Deutschland, 2012.

\bibitem{SCHOENBERG19731}
I.J. Schoenberg and A.~Sharma.
\newblock Cardinal interpolation and spline fucntions v. the b-splines for
  cardinal hermite interpolation.
\newblock {\em Linear Algebra and its Applications}, 7(1):1--42, 1973.

\bibitem{SS}
S.~Suhr.
\newblock {\em Variational Methods for Combined Image and Motion Estimation}.
\newblock PhD thesis, Fachbereich Mathematik und Informatik, Westf{\" a}lische
  Wilhelms-Universit{\" a}t M{\" u}nster, Deutschland, 2015.

\bibitem{T}
J.P. Thirion.
\newblock Non-rigid matching using demons.
\newblock In {\em Computer Vision and Pattern Recognition, 1996. Proceedings
  CVPR '96, 1996 IEEE Computer Society Conference on}, pages 245--251, 1996.

\bibitem{MS}
M.~Unser.
\newblock Splines: A perfect fit for signal and image processing.
\newblock {\em IEEE Signal Processing Magazine}, 16(6):22--38, 1999.

\bibitem{CRV}
C.R. Vogel.
\newblock {\em Computational Methods for Inverse Problem}.
\newblock SIAM, Philadelphia, 2002.

\bibitem{ZYHT}
L.~Zhu, Y.~Yang, S.~Haker, and A.~Tannenbaum.
\newblock An image morphing technique based on optimal mass preserving mapping.
\newblock {\em IEEE Transactions on Image Processing}, 16(6), 2007.

\bibitem{ZSTB2010}
Y.~Zhu, E.~Sifakis, J.~Teran, and A.~Brandt.
\newblock An efficient multigrid method for the simulation of high-resolution
  elastic solids.
\newblock {\em ACM Transactions on Graphics}, 29(16), 2010.

\end{thebibliography}

\end{document}